\documentclass[12pt,a4paper]{article}

\usepackage{a4wide}
\usepackage{epsfig}
\usepackage{authblk}
\usepackage{amsmath}
\usepackage{amssymb}
\usepackage{graphicx}
\usepackage{amsthm,dsfont,amsfonts,fancyhdr}

\usepackage[running]{lineno}

\usepackage{pst-node}

\usepackage{tikz}
\usepackage{enumerate}

\usepackage[bookmarks=false,hyperfootnotes=false]{hyperref}

\usepackage{comment}

\newtheorem{propo}{Proposition}[section]

\newtheorem{lemma}[propo]{Lemma}

\newtheorem{theo}[propo]{Theorem}

\newtheorem{remar}[propo]{Remark}

\newtheorem{prob}[propo]{Problem}

\newtheorem{cor}[propo]{Corollary}
\newtheorem{prop}[propo]{Proposition}
\newtheorem{rem}[propo]{Remark}

\newcommand{\bl}{\begin{lemma}}
\newcommand{\el}{\end{lemma}}

\makeatletter
\newcommand\blfootnote[1]{%
  \begingroup
  \renewcommand\thefootnote{}%
  \protected@xdef\@thefnmark{}
  \Hy@raisedlink{\footnotetext{#1}}
  \endgroup
}
\makeatother

\usepackage{indentfirst,latexsym,bm}

\def\K{{\rm K}}

\begin{document}

\title{Amply regular graphs with $\mu$ close to half the valency and group divisible designs}

\author[a,b]{Wei Jin}
\author[c,d]{Jack H. Koolen\footnote{Jack H. Koolen is the corresponding author.}}
\author[c]{Chenhui Lv}

\affil[a]{\footnotesize{School of Mathematics and Computational Science, Key Laboratory of Intelligent Computing and Information Processing of Ministry of Education, Xiangtan University, Xiangtan, Hunan, 411105, P.R.China}}

\affil[b]{\footnotesize{School of Statistics and Data Science, 
Jiangxi University of Finance and Economics, Nanchang, Jiangxi, 330013, P.R.China}}

\affil[c]{\footnotesize{School of Mathematical Sciences, University of Science and Technology of China, Hefei, Anhui, 230026, P.R.China}}

\affil[d]{\footnotesize{CAS Wu Wen-Tsun Key Laboratory of Mathematics, University of Science and Technology of China, Hefei, Anhui, 230026, P.R.China}}


\maketitle

\blfootnote{\small E-mail addresses: {\tt jinweipei82@163.com} (W. Jin), {\tt koolen@ustc.edu.cn} (J.H. Koolen), {\tt lch1994@mail.ustc.edu.cn} (C. Lv)}

\begin{abstract}
In this paper, we classify connected amply regular graphs with diameter $d \geq 4$ and parameters $(v, k, \lambda, \mu)$ satisfying $\mu = \frac{k-1}{2}$, where $k\geq 5$ is odd. 
We prove that such a graph must be exactly one of the following: the $5$-cube, the graph $\K_2 \square \Lambda$, where $\Lambda$ is the unique bipartite $(0,2)$-graph on $14$ vertices, or the point--block incidence graph of a group divisible design with the dual property, namely a $GDDDP\left(2, k+1;\, k;\, 0, \frac{k-1}{2}\right)$.
For the last family, we give equivalent characterizations in terms of bipartite $Q$-regular graphs and relation graphs of symmetric association schemes with five classes.  
Furthermore, we present constructions of such amply regular graphs, yielding infinite families of examples derived from Paley graphs, Peisert graphs, and Paley digraphs.
\end{abstract}

\vspace{2mm}

 \hspace{-17pt}{\bf Keywords:} amply regular graph, group divisible design, association scheme.

 \hspace{-17pt}\textbf{Mathematics Subject Classification (2020):} 05E30, 	05B05

\section{Introduction}
An \emph{amply regular graph} with parameters $(v, k, \lambda, \mu)$ is a $k$-regular graph on $v$ vertices such that every pair of adjacent vertices has exactly $\lambda$ common neighbors, and every pair of vertices at distance two has exactly $\mu$ common neighbors.

By~\cite[Theorem~1.9.3]{BCN}, for an amply regular graph with diameter $d \geq 4$ we have $\mu \leq \frac{k}{2}$, with equality if and only if the graph is a polygon or a Hadamard graph. 
For amply regular graphs with diameter $3$, however, the above bound does not hold. A counterexample is the distance-$2$ graph of the Gosset graph, which is a distance-regular graph with intersection array $\{27,16,1;1,16,27\}$.

For the case where the diameter is $3$,~\cite{BCN} posed the following question.
\begin{prob}[cf.~{\cite[p.~178]{BCN}}]
Is it true that for a distance-regular graph with $d \geq 3$ and $a_2>0$, we have 
$\mu \leq a_2 + \frac{\lambda}{2}$, except for the distance-$2$ graph of the Gosset graph whose intersection array is $\{27,16,1;1,16,27\}$? 
\end{prob}

For an upper bound on $\lambda$, Lemma~\ref{b1geqk} shows that if $d \geq 3$, then $\lambda \leq \frac{2}{3}(k-2)$.

In this paper, we provide a classification of amply regular graphs with diameter $d \geq 4$ and parameters $\mu=\frac{k-1}{2}$. Except for two exceptional graphs, all such graphs are point--block incidence graphs of certain group divisible designs.
We further present two constructions of such graphs.

Our interest in studying this class of graphs arises from another work \cite{JKL-2025-2}, in which we characterize $2$-distance-transitive graphs of small valency that are locally primitive. 
In that context, it is necessary to show that every amply regular graph with diameter $4$ and parameters $k = 7$ and $\mu = 3$ has $32$ vertices, which we prove in Corollary~\ref{valency7}. 
During the analysis of this case, we observed that the method extends naturally to the more general setting of amply regular graphs with diameter $d \geq 4$ and $\mu = \frac{k - 1}{2}$.

Let $\K_2$ denote the complete graph on two vertices (that is, a single edge), and $\square$ denotes the Cartesian product of graphs. 
Our main result is the following.

\begin{theo}\label{ARmain}
Let $\Gamma$ be a connected amply regular graph with diameter $d \geq 4$ and parameters $(v, k, \lambda, \mu)$, where $\mu = \frac{k - 1}{2}$ and $k \geq 5$ is odd.  
Let $\Lambda$ be the unique bipartite $(0,2)$-graph with $14$ vertices and valency $4$, which is the point--block incidence graph of the square $2$-$(7,4,2)$ design.
Then exactly one of the following holds:
\begin{itemize}
\item[(1)] $\Gamma$ is the $5$-cube.
\item[(2)] $\Gamma$ is the graph $\K_2 \square \Lambda$, which has diameter $4$ and parameters $(28, 5, 0, 2)$.
\item[(3)] $\Gamma$ is the point--block incidence graph of a $GDDDP\left(2, k+1;\, k;\, 0, \frac{k-1}{2}\right)$.
\end{itemize}
\end{theo}

\begin{remar}
\begin{itemize}
    \item[(1)]
    For a graph $\Gamma$ satisfying Theorem~\ref{ARmain}(3), we give equivalent characterizations in Theorem~\ref{thm:equivalent}. On the one hand, $\Gamma$ is a bipartite $Q$-regular graph with distribution diagram shown in Figure~\ref{fig:D41}. On the other hand, $\Gamma$ is a relation graph with respect to a relation $R$ of a symmetric association scheme $(X, \mathcal{R})$ with $5$ classes, such that the distribution diagram of $(X, \mathcal{R})$ with respect to the relation $R$ is as in Figure~\ref{fig:D41}.
\item[(2)]
    In Section~\ref{4-constructions}, we present two constructions of amply regular graphs with diameter $d = 4$ and parameters $(4n+4, n, 0, \frac{n - 1}{2})$. We further apply these two constructions to obtain three infinite families of examples arising from the Paley graphs, the Peisert graphs, and the Paley digraphs, respectively.
By Theorem~\ref{ARmain}, all the resulting graphs are the point--block incidence graphs of a $GDDDP\left(2, q+1;\, q;\, 0, \frac{q-1}{2}\right)$, where $q$ is a prime power such that $q \equiv 1 \pmod{4}$ or $q \equiv 3 \pmod{4}$.
It remains unclear whether these two constructions yield all such $GDDDP\left(2, q+1;\, q;\, 0, \frac{q-1}{2}\right)$. We also do not know whether these designs are new.
\end{itemize}
\end{remar}

\begin{figure}[ht]
    \centering
    \includegraphics[width=\textwidth]{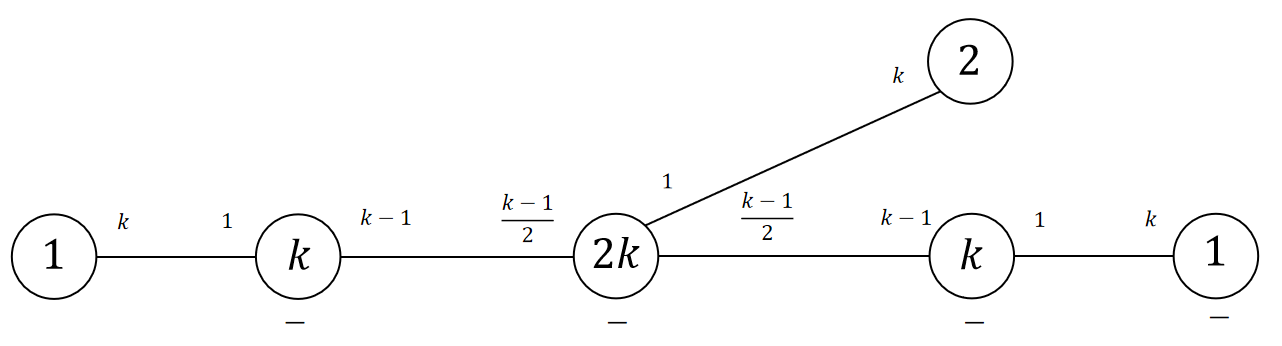} 
    \caption{The distribution diagram}
        \label{fig:D41}
\end{figure}

Koolen and Park~\cite{KP2012DCC} showed that distance-regular graphs with valency $k \geq 3$, diameter $d \geq 4$, and $c_2 > \frac{k}{3}$ are either Hadamard graphs or the $5$-cube.
Qiao, Park and Koolen~\cite{QPK2019EUJC} generalized this result by showing that a $2$-walk-regular graph with valency $k \geq 3$, diameter $d \geq 4$, and $c_2 > \frac{k}{3}$ is either a Hadamard graph, the point--block incidence graph of a $GDDDP\left(2, m;\, k;\, 0, \lambda_2\right)$, or the $5$-cube.
We believe that, apart from the Hadamard graphs and the bipartite amply regular graphs with $\mu = \frac{k-1}{2}$, there are only finitely many amply regular graphs with valency $k \geq 3$, diameter $d \geq 4$, and $\mu > \frac{k}{3}$.

Outline of the paper: In the next section, we provide preliminaries and definitions. In Section~\ref{4-constructions}, we present two constructions of amply regular graphs with diameter $d \geq 4$ and $\mu = \frac{k - 1}{2}$. We further apply these two constructions to obtain three families of examples arising from the Paley graphs, the Peisert graphs, and the Paley digraphs, respectively. For a graph $\Gamma$ satisfying Theorem~\ref{ARmain}(3), we give equivalent characterizations in Theorem~\ref{thm:equivalent}.
In Section~\ref{sec:main}, we characterize amply regular graphs with diameter $d \geq 4$ and $\mu = \frac{k - 1}{2}$, and prove Theorem~\ref{ARmain}.

\section{Preliminaries and definitions}\label{Preliminaries}

\subsection{Basic definitions in graph theory}
For a graph $\Gamma$, we denote its \emph{vertex set} by $V(\Gamma)$, its \emph{edge set} by $E(\Gamma)$. The \emph{order} of $\Gamma$ is the number of its vertices, i.e., $|V(\Gamma)|$. For two vertices $x, y \in V(\Gamma)$, if there is an edge between them, we say that $x$ is adjacent to $y$, or equivalently, that $x$ and $y$ are neighbors; this is denoted by $x \sim y$.

The \emph{distance} $d(x, y) = d_\Gamma(x, y)$ between two vertices $x, y \in V(\Gamma)$ is the length of a shortest path connecting them in $\Gamma$. The \emph{diameter} $d = d(\Gamma)$ of $\Gamma$ is the maximum distance among all pairs of vertices. For each $x \in V(\Gamma)$ and $0 \leq i \leq d$, let $\Gamma_i(x)$ denote the set of vertices at distance $i$ from $x$. Additionally, we define $\Gamma_{-1}(x) = \Gamma_{d+1}(x) = \emptyset$. For convenience, we abbreviate $\Gamma_1(x)$ as $\Gamma(x)$. The \emph{valency} of a vertex $x$ in $\Gamma$ is the cardinality $d(x) = d_\Gamma(x) := |\Gamma(x)|$. A graph $\Gamma$ is called \emph{regular} of valency $k$ if $|\Gamma(x)| = k$ for all $x \in V(\Gamma)$. For each pair of vertices $x, y \in V(\Gamma)$, we denote $\Gamma(x, y) := \Gamma(x) \cap \Gamma(y)$.

For a pair of vertices $x, y \in V(\Gamma)$ at distance $i$ (with $0 \leq i \leq d$), define the vertex sets $C_i(x, y) := \Gamma_{i-1}(x) \cap \Gamma(y)$, $A_i(x, y) := \Gamma_i(x) \cap \Gamma(y)$, and $B_i(x, y) := \Gamma_{i+1}(x) \cap \Gamma(y)$. Let $c_i(x, y) := |C_i(x, y)|$, $a_i(x, y) := |A_i(x, y)|$, and $b_i(x, y) := |B_i(x, y)|$.

Let $\K_m$ denote the \emph{complete graph} on $m$ vertices. A \emph{clique} of a graph $\Gamma$ is a set of mutually adjacent vertices of $\Gamma$.

Given two graphs $\Gamma$ and $\Delta$, the \emph{Cartesian product} $\Gamma \square \Delta$ is the graph whose vertex set is $V(\Gamma) \times V(\Delta)$, where two vertices $(x, y)$ and $(x', y')$ are adjacent if and only if either $x = x'$ and $y \sim y'$, or $x \sim x'$ and $y = y'$.

Let $\Gamma$ be a connected graph of diameter $d$. We define $\Gamma_d$ to be the graph on the same vertex set as $\Gamma$, where two vertices are adjacent if and only if they are at distance $d$ in $\Gamma$. The graph $\Gamma$ is called \emph{antipodal} if $d \geq 2$ and $\Gamma_d$ is a disjoint union of cliques. 
These cliques are called the \emph{fibres} of $\Gamma$. From an antipodal graph $\Gamma$, we may construct a smaller graph, called its \emph{folded graph} $\overline{\Gamma}$. 
The vertices of $\overline{\Gamma}$ are the fibres of $\Gamma$, and two distinct fibres are adjacent whenever there exists an edge in $\Gamma$ joining a vertex in one fibre to a vertex in the other.

The \emph{bipartite double} of a graph $\Gamma$ is defined as the graph with vertex set $V(\Gamma) \times \{0, 1\}$, where two vertices $(x, i)$ and $(y, j)$ are adjacent if and only if $x \sim y$ in $\Gamma$ and $i \neq j$. For simplicity, we write $(x, 0)$ as $x^+$ and $(x, 1)$ as $x^-$.

\subsection{Equitable partitions and distribution diagrams}
A partition $\mathcal{B} = \{B_1, B_2, \ldots, B_n\}$ of $V(\Gamma)$ is \emph{equitable} if there exist non-negative integers $q_{ij}$ ($1 \leq i, j \leq n$) such that every vertex in $B_i$ has exactly $q_{ij}$ neighbors in $B_j$. The matrix $Q = (q_{ij})_{1 \leq i,j \leq n}$ is called the \emph{quotient matrix} of $\mathcal{B}$. When $\mathcal{B}$ is equitable, the \emph{distribution diagram} of $\Gamma$ with respect to $\mathcal{B}$ is a diagram in which each block $B_i$ is represented by a balloon. We draw a line from the balloon representing $B_i$ to that of $B_j$ if $q_{ij} > 0$, placing the number $q_{ij}$ above the line near $B_i$. Inside the balloon we write $p_i := |B_i|$, and below it we write $q_{ii}$. If $q_{ii} = 0$, we write `$-$' instead of $0$.
Since there is a clear one-to-one correspondence between the quotient matrix and the distribution diagram, we will use the two terms interchangeably.

Let $\Gamma$ be a graph, and let $Q$ be the quotient matrix (equivalently, the distribution diagram) of an equitable partition of $\Gamma$ in which one of the parts has cardinality $1$.
We say that $\Gamma$ is \emph{$Q$-regular} if, for each vertex $x \in V(\Gamma)$, there exists an equitable partition $\mathcal{B}$ with quotient matrix $Q$ such that $\{x\}$ is a part of $\mathcal{B}$. 
Natural examples of $Q$-regular graphs include vertex-transitive graphs and distance-regular graphs. In the latter case, the quotient matrix $Q$ can be chosen as the intersection matrix.

Even if $\mathcal{B}$ is not an equitable partition of $V(\Gamma)$, or when its equitability is not known, we may still refer to the distribution diagram of $\Gamma$ with respect to $\mathcal{B}$ for convenience. However, in such cases, we include $q_{ij}$ values in the diagram only when they are well-defined.

\subsection{Amply regular graphs and distance-regular graphs}
An \emph{amply regular graph} with parameters $(v, k, \lambda, \mu)$ is a $k$-regular graph on $v$ vertices such that every pair of adjacent vertices has exactly $\lambda$ common neighbors, and every pair of vertices at distance two has exactly $\mu$ common neighbors.

A \emph{strongly regular graph} with parameters $(v, k, \lambda, \mu)$ is a regular graph of order $v$ and valency $k$ such that every pair of adjacent vertices has exactly $\lambda$ common neighbors, and every pair of non-adjacent vertices has exactly $\mu$ common neighbors.  For further properties of strongly regular graphs, see~\cite{BM-2022} and~\cite[Chapter~10]{GR}.

A \emph{sesqui-regular graph} with parameters $(v, k, \mu)$ is a regular graph of order $v$ and valency $k$ such that every pair of vertices at distance $2$ has exactly $\mu$ common neighbors. 

A \emph{$(0,2)$-graph} is a connected graph in which each pair of distinct vertices have either $0$ or $2$ common neighbours.
According to~\cite[Table~1]{Brouwer-2006}, there exists a unique bipartite $(0,2)$-graph with $14$ vertices and valency $4$, namely the point--block incidence graph of a square $2$-$(7,4,2)$ design.
Throughout this paper, we denote this $(0,2)$-graph by $\Lambda$.

A connected graph $\Gamma$ with diameter $d$ is called \emph{distance-regular} if there exist integers $b_i$ and $c_i$ for $0 \leq i \leq d$ such that, for each pair of vertices $x, y \in V(\Gamma)$ with $d(x, y) = i$, the vertex $y$ has exactly $c_i$ neighbors in $\Gamma_{i-1}(x)$ and $b_i$ neighbors in $\Gamma_{i+1}(x)$, where $b_d = c_0 = 0$ (cf.~\cite[p.126]{BCN}).  In this case, the set $C_i(x, y)$ contains exactly $c_i$ vertices for $1 \leq i \leq d$, and $B_i(x, y)$ contains exactly $b_i$ vertices for $0 \leq i \leq d-1$. In particular, a distance-regular graph  is regular of valency $k := b_0$, and for notational convenience, we define $a_i := k - b_i - c_i$. The constants $a_i$, $b_i$, and $c_i$ are called the \emph{intersection numbers}. The array $(b_0 = k, b_1, \dots, b_{d-1}; c_1, \dots, c_d)$ is referred to as the \emph{intersection array}.

Note that a distance-regular graph with diameter $2$ is exactly a strongly regular graph, and every distance-regular graph is, in particular, an amply regular graph with parameters $(v, b_0, a_1, c_2)$. In this paper, we adopt the notation $k$, $\lambda$, and $\mu$ when discussing strongly regular and amply regular graphs, rather than using the intersection number notation $a_i$, $b_i$, and $c_i$.

A distance-regular graph with diameter $d$ has exactly $d+1$ distinct eigenvalues, which can be computed directly from the intersection numbers; these are precisely the eigenvalues of the corresponding \emph{intersection matrix}. For more details, see~\cite[p.~128]{BCN}.

The following lemma gives constraints on the parameters of amply regular graphs.
It was proved for distance-regular graphs in~\cite[Lemma~3.1]{BK-2009}; see also~\cite[Lemma~3]{KP-2012}.
Here we replace distance-regular graphs by amply regular graphs, since the same argument applies. For completeness, we include the proof below, which follows that of~\cite[Lemma~3.1]{BK-2009}.

\begin{lemma} [{cf.~\cite[Lemma~3.1]{BK-2009}}]\label{b1geqk} 
Let $\Gamma$ be a connected amply regular graph with parameters $(v,k,\lambda,\mu)$
and diameter $d \geq 3$. Let $b_1 = k - \lambda - 1$.
Then $b_1 \geq \frac{1}{3}(k+1)$ holds.
\end{lemma}
\proof
First note that, for a pair of adjacent vertices $x,y\in\Gamma$, we have
$|\Gamma_2(x)\cap\Gamma(y)|=b_1$.
Suppose that $x_1 \sim x_2 \sim x_3 \sim x_4$ is a shortest path.

If $b_1 < \frac{1}{3}(k+1)$, then $\lambda > 2(b_1-1)$, and hence not all common neighbours of $x_2$ and $x_3$ are nonadjacent to $x_1$ or to $x_4$.
Therefore, $x_1$ and $x_4$ have a common neighbour, a contradiction.
\qed

\subsection{Taylor graphs}

A distance-regular graph $\Gamma$ with intersection array $(k, \mu, 1;\ 1, \mu, k)$ is called a \emph{Taylor graph}. Such a graph has diameter $3$ and order $2(k + 1)$. Moreover, for each vertex $x \in V(\Gamma)$, we have $|\Gamma(x)| = |\Gamma_2(x)| = k$ and $|\Gamma_3(x)| = 1$.

Given a graph $\Gamma$ with vertex set $V(\Gamma)$, its \emph{Taylor double} is the graph with vertex set $\{x^\varepsilon \mid x \in V(\Gamma),\ \varepsilon = \pm 1\}$, where two distinct vertices $x^\delta$ and $y^\varepsilon$ are adjacent (for $x \neq y$) if and only if $\delta\varepsilon = 1$ when $x \sim y$, and $\delta\varepsilon = -1$ when $x \not\sim y$. 

Given a strongly regular graph $\Delta$ with $v$ vertices satisfying $k = 2\mu$, its \emph{Taylor extension} is the Taylor double of the graph $\{\infty\} + \Delta$, where $\infty$ is a new vertex adjacent to every vertex of $\Delta$. The resulting graph is a Taylor graph on $2(v + 1)$ vertices with intersection array $\{v, v - k - 1, 1;\ 1, v - k - 1, v\}$ (see~\cite[p.19]{BM-2022}).

\subsection{Association schemes}
Let $X$ be a finite set, and let $\mathbb{C}^{X \times X}$ denote the set of complex matrices with rows and columns indexed by $X$. Let $\mathcal{R} = \{R_0, R_1, \ldots, R_n\}$ be a collection of non-empty subsets of $X \times X$, where each $R_i$ $(0 \leq i \leq n)$ is called a \emph{relation}. For each $i$, the \emph{relation graph} $\Gamma_i^{\mathcal{R}} := (X, R_i)$ corresponding to the relation $R_i$ is (in general) a directed graph with vertex set $X$ and edge set $R_i$. Let $F_i$ be the adjacency matrix of the graph $\Gamma_i^{\mathcal{R}}$. The pair $(X, \mathcal{R})$ is called an \emph{association scheme} with $n$ classes if the following conditions hold:

\begin{itemize}
    \item[(i)] $F_0 = I$, the identity matrix;
    \item[(ii)] $\sum_{i=0}^n F_i = J$, the all-ones matrix;
    \item[(iii)] $F_i^\top \in \{F_0, F_1, \ldots, F_n\}$ for all $0 \leq i \leq n$;
    \item[(iv)] $F_i F_j$ is a linear combination of $F_0, F_1, \ldots, F_n$ for all $0 \leq i, j \leq n$.
\end{itemize}

We refer to $|X|$ as the \emph{order} of the association scheme $(X, \mathcal{R})$. The vector space $\mathbb{A}$ spanned by $\{F_0, F_1, \ldots, F_n\}$ is called the \emph{Bose–Mesner algebra} of $(X, \mathcal{R})$. The scheme $(X, \mathcal{R})$ is said to be \emph{commutative} if $\mathbb{A}$ is commutative, and \emph{symmetric} if each matrix $F_i$ $(0 \leq i \leq n)$ is symmetric. Every symmetric association scheme is commutative.

Suppose $(X, \mathcal{R})$ is a symmetric association scheme with $n$ classes. Then each relation graph $\Gamma_i^{\mathcal{R}}$ ($1 \leq i \leq n$) is an undirected graph.  
Fix a vertex $x$ of $\Gamma_i^{\mathcal{R}}$. Then the partition $\{B_0, B_1, \ldots, B_n\}$ of $V(\Gamma_i^{\mathcal{R}})$, defined by $ B_j := \{y \in X \mid (x, y) \in R_j\}$, for $ 0 \leq j \leq n$, is an equitable partition. The \emph{distribution diagram} of the symmetric association scheme $(X, \mathcal{R})$ with respect to the relation $R_i$ is defined as the distribution diagram of the relation graph $\Gamma_i^{\mathcal{R}}$ with respect to the equitable partition $\{B_0, B_1, \ldots, B_n\}$.
For more on association schemes, see~\cite{BI-1984}.

\subsection{Group divisible designs} 

An \emph{incidence structure} $\mathcal{I} = (\mathcal{P}, \mathcal{B}, I)$ consists of a set $\mathcal{P}$ of points, a set $\mathcal{B}$ of blocks (disjoint from $\mathcal{P}$), and a relation $I \subseteq \mathcal{P} \times \mathcal{B}$ called \emph{incidence}.  
If $(p, B) \in I$, then we say that the point $p$ and the block $B$ are \emph{incident}.  
In most contexts, we regard each block $B$ as a subset of $\mathcal{P}$.  
Given an incidence structure $\mathcal{I} = (\mathcal{P}, \mathcal{B}, I)$, its \emph{dual incidence structure} is defined as $\mathcal{I}^* = (\mathcal{B}, \mathcal{P}, I^*)$, where $I^* = \{(B, p) \mid (p, B) \in I\}$.  
The \emph{point--block incidence graph} $\Gamma(\mathcal{I})$ of $\mathcal{I}$ is the graph with vertex set $\mathcal{P} \cup \mathcal{B}$, where two vertices are adjacent if and only if they are incident.  
Note that the point--block incidence graph of an incidence structure is a bipartite graph.

A \emph{group divisible design} $\mathcal{D} = (\mathcal{P}, \mathcal{G}, \mathcal{B})$ with parameters $(n, m; k; \lambda_1, \lambda_2)$, denoted by $GDD(n, m; k; \lambda_1, \lambda_2)$, consists of a set $\mathcal{P}$ of points, a partition $\mathcal{G}$ of $\mathcal{P}$ into $m$ subsets of size $n$ (called \emph{groups}), and a collection $\mathcal{B}$ of $k$-subsets of $\mathcal{P}$ (called \emph{blocks}), such that:
\begin{itemize}
    \item[(1)] each pair of points from the same group appears in exactly $\lambda_1$ blocks, and
    \item[(2)] each pair of points from different groups appears in exactly $\lambda_2$ blocks.
\end{itemize}
The triple $\mathcal{I} = (\mathcal{P}, \mathcal{B}, I)$, with the natural incidence relation $I$, forms an incidence structure, and we consider its dual $\mathcal{I}^* = (\mathcal{B}, \mathcal{P}, I^*)$.  
If there exists a partition $\mathcal{G}'$ of $\mathcal{B}$ such that the triple $(\mathcal{B}, \mathcal{G}', \mathcal{P})$ is a $GDD(n, m; k; \lambda_1, \lambda_2)$, then we say that $\mathcal{D}$ is a \emph{group divisible design with the dual property}, with parameters $(n, m; k; \lambda_1, \lambda_2)$, and we denote it by $GDDDP(n, m; k; \lambda_1, \lambda_2)$.

\section{Some families of amply regular graphs with  parameter $\mu = \frac{k - 1}{2}$}\label{4-constructions}
In this section, we present several examples of amply regular graphs with parameter $\mu = \frac{k - 1}{2}$. Some properties established later in Section~\ref{sec:main} will be used in the discussion of these examples. However, the proofs of those properties are independent of the material in the present section.

Let $q \geq 5$ be an odd prime power. We present three families of amply regular graphs with diameter $d = 4$ and parameters $(4q+4, q, 0, \frac{q - 1}{2})$, constructed from the Paley graphs, the Peisert graphs, and the Paley digraphs, respectively.
By Theorem~\ref{ARmain}, all the resulting graphs are the point--block incidence graphs of a $GDDDP\left(2, q+1;\, q;\, 0, \frac{q-1}{2}\right)$.
By Theorem~\ref{thm:equivalent}, all the resulting graphs are bipartite $Q$-regular graphs with distribution diagram shown in Figure~\ref{fig:D41}.

\subsection{Amply regular graphs with  parameter $\mu = \frac{k - 1}{2}$ from conference graphs}\label{app-A}

 The conference graphs are strongly regular graphs with parameters
\[
(v, k, \lambda, \mu) = \left(n, \frac{n - 1}{2}, \frac{n - 5}{4}, \frac{n - 1}{4}\right),
\]
where $n \equiv 1 \pmod{4}$ is a positive integer.
Their Taylor extension $\Sigma$ is a distance-regular graph on $2(n + 1)$ vertices with parameters
\[
k(\Sigma) = n, \quad b_1(\Sigma) = c_2(\Sigma) = \frac{n - 1}{2}, \quad \text{and} \quad a_1(\Sigma) = \frac{n - 1}{2}.
\]
The bipartite double of $\Sigma$, denoted by $\Gamma$, is a connected amply regular graph with diameter $d = 4$ and parameters
\[
(v, k, \lambda, \mu) =\left(4(n + 1),\, n,\, 0,\, \frac{n - 1}{2}\right),
\]
which admits an equitable partition with distribution diagram shown in Figure~\ref{fig:D41}.

Let $q$ be a prime power such that $q \equiv 1 \pmod{4}$. The \emph{Paley graph} $P(q)$ is defined as the graph with vertex set $\mathbb{F}_q$, where two distinct vertices $x, y \in \mathbb{F}_q$ are adjacent if and only if $x - y$ is a nonzero square in $\mathbb{F}_q$. The congruence condition guarantees that $-1$ is a square in $\mathbb{F}_q$, ensuring the graph is undirected. The graph $P(q)$ is a strongly regular graph with parameters $(v, k, \lambda, \mu) = \left(q, \frac{q - 1}{2}, \frac{q - 5}{4}, \frac{q - 1}{4} \right)$, and hence is a conference graph.

Let $q = p^r \geq 5$ be a prime power such that $p \equiv 3 \pmod{4}$ and $r$ is even. Note that, by the choice of $p$ and $r$,  $q \equiv 1 \pmod{4}$.
The \emph{Peisert graph} $P^*(q)$ is defined as the graph with vertex set $\mathbb{F}_q$, where two distinct vertices $x, y \in \mathbb{F}_q$ are adjacent if and only if $x - y$ belongs to the set $M = \{w^j : j \equiv 0,1 \pmod{4}\}$,
where $w$ is a primitive root of $\mathbb{F}_q$. 
Peisert graphs are strongly regular and share the same parameters as Paley graphs when defined on the same number of vertices. More precisely, the Peisert graph $P^*(q)$ has parameters $(v, k, \lambda, \mu) = \left(q, \frac{q - 1}{2}, \frac{q - 5}{4}, \frac{q - 1}{4} \right)$, and thus also qualifies as a conference graph. For more information on the Peisert graphs, we refer the reader to~\cite{Peisert-2001}.

Using the construction method described above, we can therefore derive families of amply regular graphs with parameter $\mu = \frac{k-1}{2}$ from the Paley graphs and the Peisert graphs. More precisely, we obtain the following proposition, which follows immediately from the bipartite double construction together with Theorem~\ref{diameter4}.

\begin{prop}\label{prop:paley}
Let $\Delta$ be either the Paley graph $P(q)$, where $q$ is a prime power satisfying $q\equiv 1\pmod{4}$, or the Peisert graph $P^*(q)$, where $q=p^r\ge 5$ with $p\equiv 3\pmod{4}$ and $r$ even.
Let $\Sigma$ be the Taylor extension of $\Delta$, and let $\Gamma$ be the bipartite double of $\Sigma$. 
Then $\Gamma$ is a connected bipartite amply regular graph with diameter $d = 4$ and parameters $(v, k, \lambda, \mu) = \left(4(q + 1),\, q,\, 0,\, \frac{q - 1}{2}\right)$, and $\Gamma$ is a $Q$-regular graph with distribution diagram shown in Figure~\ref{fig:D41}.
\end{prop}

\subsection{Amply regular graphs with parameter $\mu=\frac{k-1}{2}$ from non-symmetric association schemes with $2$ classes}\label{app-C}

Let $\{A_0 = I_n, A_1, A_2\}$ be the set of adjacency matrices of a non-symmetric association scheme with $2$ classes and order $n$.
Ikuta and Munemasa~\cite{IM-2022} proposed a method to construct non-symmetric association schemes  with $3$ classes from those  with $2$ classes, as follows. Define
\[
C_0 = I_{2(n+1)},\qquad
C_1 = \begin{bmatrix}
0 & \mathbf{1} & 0 & 0 \\
0 & A_1 & A_2 & \mathbf{1}^T \\
\mathbf{1}^T & A_2 & A_1 & 0 \\
0 & 0 & \mathbf{1} & 0
\end{bmatrix}, \qquad
C_2 = C_1^T, \qquad
C_3 = J - C_0 - C_1 - C_2,
\]
where $\mathbf{1}$ denotes the all-one row vector of length $n$. By~\cite[Theorem~4]{IM-2022}, the matrices $\{C_0, C_1, C_2, C_3\}$ form the adjacency matrices of a non-symmetric association scheme with $3$ classes.
Moreover, from~\cite[Lemma~3]{IM-2022}, the following identity holds:
\begin{equation}\label{eq:C1C2}
C_1 C_2 = C_2 C_1 = n C_0 + \frac{n - 1}{2}(C_1 + C_2).
\end{equation}

We now construct an amply regular graph with valency $k = n$ and parameter $\mu = \frac{n - 1}{2}$. 

\begin{prop}\label{prop:amply-from-ass}
Let $\{A_0 = I_n, A_1, A_2\}$ be the set of adjacency matrices of a non-symmetric association scheme with $2$ classes and order $n \geq 5$.
Let $\{C_0, C_1, C_2, C_3\}$ be defined as above. 
Let
\[
B = \begin{bmatrix}
0 & C_1 \\
C_1^T & 0
\end{bmatrix}
\]
be the adjacency matrix of an undirected graph $\Gamma$.
Then $\Gamma$ is a connected bipartite amply regular graph with diameter $d = 4$ and parameters
$(v, k, \lambda, \mu) = \left(4(n + 1),\, n,\, 0,\, \frac{n - 1}{2}\right)$, and $\Gamma$ is a $Q$-regular graph with distribution diagram shown in Figure~\ref{fig:D41}.
\end{prop}
\proof
As $A_1 + A_1^T = A_2 + A_2^T = A_1 + A_2 = J - I$, it follows that 
\begin{align}
C_1 + C_2 &= \begin{bmatrix}
0 & \mathbf{1} & 0 & 0 \\
0 & A_1 & A_2 & \mathbf{1}^T \\
\mathbf{1}^T & A_2 & A_1 & 0 \\
0 & 0 & \mathbf{1} & 0
\end{bmatrix}
+
\begin{bmatrix}
0 & 0 & \mathbf{1} & 0 \\
\mathbf{1}^T & A_1^T & A_2^T & 0 \\
0 & A_2^T & A_1^T & \mathbf{1} ^T \\
0 & \mathbf{1} & 0 & 0
\end{bmatrix}  
=
\begin{bmatrix}
0 & \mathbf{1} & \mathbf{1} & 0 \\
\mathbf{1}^T & J-I & J-I & \mathbf{1}^T \\
\mathbf{1}^T & J-I & J-I & \mathbf{1} ^T \\
0 & \mathbf{1} & \mathbf{1} & 0
\end{bmatrix}.   \notag
\end{align}
Therefore, by \eqref{eq:C1C2},
\begin{align}\label{C1C2-amply}
C_1 C_2 = C_2 C_1 = n
\begin{bmatrix}
1 & 0 & 0 & 0 \\
0 & I & 0 & 0 \\
0 & 0 & I & 0 \\
0 & 0 & 0 & 1
\end{bmatrix}
+ \frac{n-1}{2}
\begin{bmatrix}
0 & \mathbf{1} & \mathbf{1} & 0 \\
\mathbf{1}^T & J-I & J-I & \mathbf{1}^T \\
\mathbf{1}^T & J-I & J-I & \mathbf{1} ^T \\
0 & \mathbf{1} & \mathbf{1} & 0
\end{bmatrix}.
\end{align}

Let $V(\Gamma) = X_1 \cup X_2$ be the bipartition of the vertex set. 
Since
\begin{align*}
B^2 = 
\begin{bmatrix}
C_1 C_1^T & 0 \\
0 & C_1^T C_1
\end{bmatrix}
= 
\begin{bmatrix}
C_1 C_2 & 0 \\
0 & C_2 C_1
\end{bmatrix},
\end{align*}
for each $x \in X_i$ ($i = 1,2$), there exists a unique $x' \in X_i$ such that $d_\Gamma(x,x') = 4$ and $d_\Gamma(x,u) = 2$ for all $u \in X_i \setminus \{x,x'\}$.
Consequently, $\Gamma$ is a connected bipartite graph with diameter $d = 4$.
Combining this with \eqref{C1C2-amply}, we conclude that $\Gamma$ is an amply regular graph with $4n+4$ vertices, valency $k = n$, and parameter $\mu = \frac{n - 1}{2}$.

As $d = 4$ and $k \geq 5$, Theorem~\ref{diameter4} implies that $\Gamma$ is a bipartite $Q$-regular graph with distribution diagram shown in Figure~\ref{fig:D41}.
\qed

Next, we apply the construction described in this subsection to obtain a family of amply regular graphs with parameter $\mu=\frac{k-1}{2}$ from the Paley digraphs.

Let $q \geq 7$ be a prime power such that $q \equiv 3 \pmod{4}$. Then the finite field $\mathbb{F}_q$ contains no square root of $-1$. Consequently, for each pair of distinct elements $a, b \in \mathbb{F}_q$, exactly one of $a - b$ and $b - a$ is a square in $\mathbb{F}_q$.
The \emph{Paley digraph} $\vec{P}(q)$ is the directed graph with vertex set $\mathbb{F}_q$, where there is an arc from $a$ to $b$ (with $a \ne b$) if and only if $b - a \in 
\Box = \{ x \in \mathbb{F}_q^* \mid x = y^2 \text{ for some } y \in \mathbb{F}_q \}$.

Let $A_1$ be the adjacency matrix of $\vec{P}(q)$, and let $A_2 = A_1^T$ be its transpose. Then $I_q + A_1 + A_2 = J$, where $I_q$ is the $q \times q$ identity matrix and $J$ is the all-ones matrix of the same order. 
It is easy to verify that the set of  matrices $\{A_0 = I_q, A_1, A_2\}$ forms a non-symmetric association scheme with $2$ classes; see, for example,~\cite[p.~49]{BI-1984}.

Using the construction method described above, we can therefore derive a family of amply regular graphs with parameter $\mu = \frac{k - 1}{2}$ from the Paley digraphs, satisfying Proposition~\ref{prop:amply-from-ass}.

\subsection{A characterization of bipartite amply regular graphs with $\mu = \frac{k-1}{2}$}

In this subsection, we give equivalent characterizations of the graphs satisfying Theorem~\ref{ARmain}(3), thereby generalizing Theorem~2 of Qiao et al.~\cite{qdk-1}.

\begin{theo}\label{thm:equivalent}
Let $\Gamma$ be a connected amply regular graph with diameter $d = 4$ and parameters $(v, k, 0, \mu)$, where $\mu = \frac{k - 1}{2}$ and $k \geq 5$ is odd. 
Then the following are equivalent:
\begin{itemize}
    \item[(1)] $\Gamma$ is a bipartite $Q$-regular graph with distribution diagram shown in Figure~\ref{fig:D41}.

    \item[(2)] $\Gamma$ is the point--block incidence graph of a $GDDDP\left(2, k+1;\, k;\, 0, \frac{k-1}{2}\right)$.
    \item[(3)] $\Gamma$ is a relation graph with respect to a relation $R$ of a symmetric association scheme $(X, \mathcal{R})$ with $5$ classes such that the distribution diagram of $(X, \mathcal{R})$ with respect to $R$ is as in Figure~\ref{fig:D41}. In particular, $\Gamma$ is a bipartite $2$-walk-regular graph of order $4k + 4$, diameter $4$, and exactly $6$ distinct eigenvalues.
    \end{itemize}
\end{theo}
\proof
{\boldmath $(1) \Rightarrow (2)$:} Fix a vertex $x \in V(\Gamma)$, and define the triple $\mathcal{D} = (\mathcal{P}, \mathcal{G}, \mathcal{B})$ as follows:
\[
\mathcal{P} := \{ y \in V(\Gamma) \mid d_{\Gamma}(x, y) \equiv 0 \pmod{2} \},
\]
\[
\mathcal{G} := \left\{ \{y_1, y_2\} \mid d_{\Gamma}(y_1, x),\ d_{\Gamma}(y_2, x) \equiv 0 \pmod{2},\ d_{\Gamma}(y_1, y_2) = 4 \right\},
\]
\[
\mathcal{B} := \left\{ \Gamma(z) \mid z \in V(\Gamma)\ \text{and}\ d_{\Gamma}(z, x) \equiv 1 \pmod{2} \right\}.
\]

For each vertex $y \in V(\Gamma)$, since $\Gamma$ is a bipartite $Q$-regular graph with distribution diagram shown in Figure~\ref{fig:D41}, it follows that $|\Gamma_4(y)| = 1$.
Thus, there exists a unique vertex $y'$ at distance $4$ from $y$.
Hence $\mathcal{G}$ is a partition of $\mathcal{P}$.
Furthermore,  $\mathcal{B}$ is a collection of $k$-subsets of $\mathcal{P}$ such that each pair of points from the same group appears in exactly $0$ blocks,
and each pair of points from different groups appears in exactly $\mu = \frac{k-1}{2}$ blocks. Thus $\mathcal{D} = (\mathcal{P}, \mathcal{G}, \mathcal{B})$ is a group divisible design with parameters $\left(2, k+1;\, k;\, 0, \frac{k-1}{2}\right)$, with $\Gamma$ as its point--block incidence graph.

Moreover, by taking a vertex at odd distance from $x$ and repeating the above construction, we obtain that $\mathcal{D}$ is a $GDDDP\left(2, k+1;\, k;\, 0, \frac{k-1}{2}\right)$.

{\boldmath $(2) \Rightarrow (3)$:} This follows immediately from \cite[Theorem~2]{qdk-1}.

{\boldmath $(3) \Rightarrow (1)$:}
Let $(X, \mathcal{R})$ be a symmetric association scheme, where $\mathcal{R} = \{R_0, R_1, \ldots, R_5\}$, and let $\Gamma$ be the relation graph with respect to the relation $R$. Then clearly $\Gamma$ is a bipartite graph.
As the distribution diagram of $(X, \mathcal{R})$ with respect to the relation $R$ is as in Figure~\ref{fig:D41}, for each vertex $x$ of $\Gamma$, the partition $\{B_0, B_1, \ldots, B_5\}$ of $V(\Gamma)$, where $B_j := \{y \mid (x,y) \in R_j\}$ for $0 \leq j \leq 5$, is equitable, and its distribution diagram is as in Figure~\ref{fig:D41}. 
Therefore, $\Gamma$ is a bipartite $Q$-regular graph with distribution diagram shown in Figure~\ref{fig:D41}.
\qed

\begin{rem}
Ikuta and Munemasa~\cite{IM-2022} constructed non-symmetric $3$-class association schemes from the directed Paley graphs. Using their construction and an analogue of the bipartite double, we construct bipartite amply regular graphs with $\mu = \frac{k - 1}{2}$. It follows from Theorem~\ref{thm:equivalent} that these graphs are the relation graphs of certain symmetric $5$-class association schemes.
\end{rem}

\section{Amply regular graphs with parameter $\mu = \frac{k-1}{2}$}\label{sec:main}

In this section, we aim to characterize amply regular graphs with diameter $d \geq 4$ and parameters $(v, k, \lambda, \mu)$, where $\mu = \frac{k - 1}{2}$ and $k \geq 5$ is odd. The main result of this section is a proof of Theorem~\ref{ARmain}.

We start with several lemmas that will be used repeatedly in this section. For convenience, we restate~\cite[Lemma 2.2]{BK-2009} below.

\begin{lemma}[cf.~{\cite[Lemma 2.2]{BK-2009}}]\label{mubound}
Let $\Gamma$ be a bipartite graph with vertex partition $V(\Gamma) = X \cup Y$, where $|X| = n$ and $|Y| = \alpha \geq 2$. Suppose that each vertex in $Y$ has, on average, $w$ neighbors in $X$. Then there exists a pair of vertices $y_1, y_2 \in Y$ that have at least $ \frac{w^2}{n} - \frac{w(n - w)}{n(\alpha - 1)} = \frac{w}{n}\left(w - \frac{n-w}{\alpha-1}\right)$ common neighbors.
\end{lemma}

In our case of interest, Lemma~\ref{mubound} can be strengthened, as stated below in Lemma~\ref{equalitycase}.

\begin{lemma}\label{equalitycase}
Let $\Gamma$ be a bipartite graph with vertex partition $V(\Gamma) = X \cup Y$. Suppose that $d(x_i) = s_i$ for $x_i \in X$, where $1 \leq i \leq |X|$, and that $d(y) = k$ for all $y \in Y$. Further assume that $\mu>0$ and every pair of distinct vertices $y_1, y_2 \in Y$ with $d(y_1, y_2) = 2$ have exactly $\mu$ common neighbors. Then we have $(|Y| - 1)\mu \geq k\left( \frac{|Y| \cdot k}{|X|} - 1 \right)$, with equality if and only if $s_1 = s_2 = \cdots = s_{|X|}$ and every pair of vertices in $Y$ has distance 2.
\end{lemma}
\proof
Let $W = \{(x, \{y_1, y_2\}) \mid x \in X, \{y_1,y_2\} \subseteq Y, y_1\neq y_2, x \sim y_i \text{ for } i = 1, 2\}$. Then we have
\[
\binom{|Y|}{2} \mu \geq |W| = \sum_{i=1}^{|X|} \binom{s_i}{2} \geq |X|  \binom{\frac{s_1 + \cdots + s_{|X|}}{|X|}}{2}  = |X| \binom{ \frac{|Y| \cdot k}{|X|}}{2} ,
\]
that is, $(|Y| - 1)\mu \geq k \left( \frac{|Y| \cdot k}{|X|} - 1 \right)$, with equality if and only if $s_1 = s_2 = \cdots = s_{|X|}$ and every pair of vertices in $Y$ has distance $2$.
\qed

The following result can be seen as a generalization of~\cite[Theorem 5.4.1]{BCN}.

\begin{lemma}\label{c3-c2}
Let $\Gamma$ be a sesqui-regular graph with parameters $(v, k, \mu)$, where $\mu \geq 2$. 
Suppose there exists a path $x_1 \sim x_2 \sim x_3 \sim x_4 \sim x_5$ such that $d(x_1,x_5)=4$ and the vertices $x_2, x_3, x_4$ are contained in an induced quadrangle. 
Then $c_3(x_2, x_5) + c_3(x_4, x_1) \geq 3\mu$.
\end{lemma}
\proof
Assume that $x_2 \sim x_3 \sim x_4 \sim y \sim x_2$ forms an induced quadrangle. 
Then we have $d(y,x_1) = d(y,x_3) = d(y,x_5) = 2$.
Since $C_2(x_3, x_5) \cup C_2(y, x_5) \subseteq C_3(x_2, x_5)$, the inclusion--exclusion principle yields 
\[
    |C_2(x_3, x_5) \cap C_2(y, x_5)| \geq |C_2(x_3, x_5)| + |C_2(y, x_5)| - |C_3(x_2, x_5)| = 2\mu - c_3(x_2, x_5). 
\]
Hence $|\Gamma(x_3) \cap \Gamma(y) \cap \Gamma(x_5)| \geq 2\mu - c_3(x_2, x_5)$.

By the same argument, $|\Gamma(x_3) \cap \Gamma(y) \cap \Gamma(x_1)| \geq 2\mu - c_3(x_4, x_1)$. 
Since $d(x_1, x_5) = 4$, the sets $\Gamma(x_3) \cap \Gamma(y) \cap \Gamma(x_5)$ and $\Gamma(x_3) \cap \Gamma(y) \cap \Gamma(x_1)$ are disjoint, and both are contained in $\Gamma(x_3) \cap \Gamma(y)$. 
Therefore, $\mu \geq (2\mu - c_3(x_2, x_5)) + (2\mu - c_3(x_4, x_1))$, which implies $c_3(x_2, x_5) + c_3(x_4, x_1) \geq 3\mu$.
\qed

The following lemma shows that, for the problem under consideration, we always have $\lambda = 0$ and $d \leq 5$.

\begin{lemma}\label{lambda=0}
Let $\Gamma$ be a connected amply regular graph with diameter $d \geq 4$ and parameters  
$(v, k, \lambda, \mu)$, where $\mu = \frac{k - 1}{2}$ and $k \geq 5$ is odd. 
Then $\lambda = 0$ and $d \leq 5$.
\end{lemma}
\proof
If $\lambda = 0$, then by~\cite[Corollary~1.9.2]{BCN}, we have $d \leq k + 4 - 2\left(\frac{k - 1}{2}\right) = 5$. Therefore, it suffices to prove that $\lambda = 0$.

First, we have $\lambda \leq \mu - 1$. Otherwise, by \cite[Theorem 1.5.5]{BCN}, we have $k > \lambda + \mu + 1 \geq 2\mu + 1 = k$, a contradiction.

In the following we assume that $\lambda > 0$.

Let $x, y, w, z, y'$ be distinct vertices such that $d_{\Gamma}(x, z) = 4$, $y \in \Gamma(x) \cap \Gamma_3(z)$, $w \in \Gamma_2(y) \cap \Gamma(z)$, and $y'$ is a common neighbor of $x$ and $y$. 
If $\Gamma(y,w) \subseteq \Gamma(y')$, then $\lambda \geq \mu + 1$, a contradiction. Thus, there exists a vertex $v \in \Gamma(y,w) \setminus \Gamma(y')$.
Now we have $d(x,v)=d(z,v)=2$. As $\mu=\frac{k-1}{2}$, we may assume that $\Gamma(v)=\Gamma(x,v)\cup \Gamma(z,v)\cup \{v'\}$ for some vertex $v' \notin \Gamma(x)\cup \Gamma(z)$. Note that there are no edges between $\Gamma(x, v)$ and $\Gamma(z, v)$. Since $|\Gamma(x, y)| = |\Gamma(y, v)| = \lambda$ and $\Gamma(y, v) \subseteq \Gamma(x, v) \cup \{v'\}$, it follows that $y \sim v'$, and hence $v' \in \Gamma(y, v)$. As $d_{\Gamma}(y', v) = 2$, the common neighbors of $y'$ and $v$ must lie in $\Gamma(x, v) \cup \{v'\}$; otherwise, we would have $d_{\Gamma}(x, z) = 3$, contradicting our assumption. Therefore, $y'$ is adjacent to $\mu$ vertices in $\Gamma(x, v) \cup \{v'\}$, which implies that $\lambda \geq \mu - 1$.

Therefore we have $\lambda = \mu - 1$. Now, each vertex in $\Gamma(x)$ has $\frac{k + 1}{2}$ neighbors in $\Gamma_2(x)$, and each vertex in $\Gamma_2(x)$ has $\frac{k - 1}{2}$ neighbors in $\Gamma(x)$. Let $k_2 = |\Gamma_2(x)|$. Counting the edges between $\Gamma(x)$ and $\Gamma_2(x)$ in two ways gives $k(k + 1) = (k - 1)k_2$. It is easy to verify that $k_2 \neq k + 2$. Therefore, we may assume $k_2 = k + i$ for some integer $i \geq 3$. Then we obtain $\frac{i}{i - 2} = k \geq 5$, which implies $10 \geq 4i$, a contradiction.
Hence, we conclude that $\lambda = 0$, completing the proof of Lemma~\ref{lambda=0}.
\qed

\vspace{0.2cm}

By Lemma~\ref{lambda=0}, we only need to consider the case $\lambda = 0$.

The following lemma indicates that, for the problem under consideration, if the amply regular graph is not bipartite, then we may instead consider its bipartite double, which is easier to handle.

\begin{lemma}[{\cite[Theorem 1.11.1]{BCN}}]\label{bipartite-double}
Let $\Gamma$ be a connected amply regular graph with parameters $(v, k, 0, \mu)$ that is not bipartite. Then its bipartite double is a connected amply regular graph with parameters $(2v, k, 0, \mu)$.
\end{lemma}

In Theorem~\ref{diameter5} below, we first consider the case when the diameter $d = 5$.

\begin{theo}\label{diameter5}
Let $\Gamma$ be a connected amply regular graph with diameter $d = 5$ and parameters $(v, k, 0, \mu)$, where $\mu = \frac{k - 1}{2}$ and $k \geq 5$ is odd. Then $\Gamma$ is the 5-cube.
\end{theo}
\proof
If $\Gamma$ is not bipartite, then by Lemmas~\ref{lambda=0} and~\ref{bipartite-double}, the bipartite double of $\Gamma$ is a connected amply regular graph with diameter $5$ and parameters $\left(2v, k, 0, \frac{k - 1}{2}\right)$. Therefore, we may first consider the case where $\Gamma$ is bipartite.

Choose a vertex $x$ with $\Gamma_5(x)\neq\emptyset$, and let $k_i = |\Gamma_i(x)| = |\{ y \in V(\Gamma) \mid d_\Gamma(x, y) = i \}|$ for $0 \leq i \leq 5$. Let $\mathcal{B} = \{\Gamma_0(x), \Gamma_1(x), \Gamma_2(x), \Gamma_3(x), \Gamma_4(x), \Gamma_5(x)\}$ be a partition of the vertex set of $\Gamma$. This yields the distribution diagram of $\Gamma$ with respect to $\mathcal{B}$, as shown in Figure~\ref{fig:AA1}. In particular, we have $k_2 = 2k$.

\begin{figure}[ht]
    \centering
    \includegraphics[width=\textwidth]{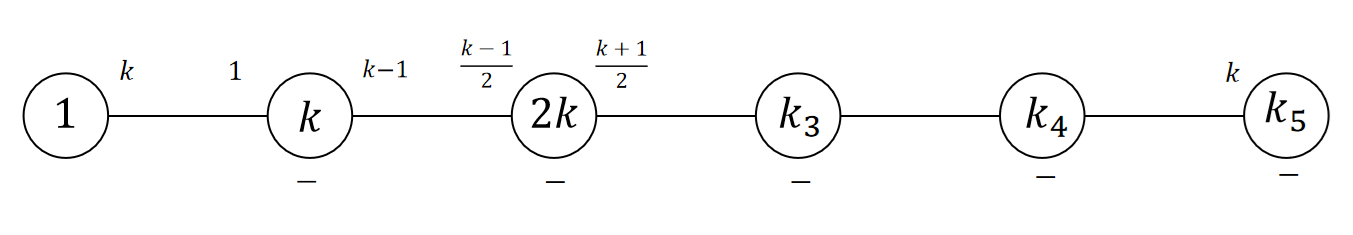} 
    \caption{The distribution diagram}
    \label{fig:AA1}
\end{figure}

For each $y \in \Gamma_3(x)$, by~\cite[Proposition~1.9.1]{BCN} we have 
$c_3(x,y) \geq c_2 + 1 = \frac{k+1}{2}$, and hence $k_3 \leq k_2 = 2k$. 
Furthermore, $b_3(x,y) \leq \frac{k-1}{2}$.

For each $z \in \Gamma_4(x)$, if $\Gamma(z) \cap \Gamma_5(x) = \emptyset$, then $c_4(x,z)=k$.
Otherwise, there exist vertices $w_i \in \Gamma_i(x)$ for $i=1,2,3,5$ such that
$w_1 \sim w_2 \sim w_3 \sim z \sim w_5$ is a path of length $4$ and $d(w_1,w_5)=4$.
As $\mu \geq 2$ and $\Gamma$ is bipartite, Lemmas~\ref{c3-c2} applies, and hence
\begin{align}\label{c3+c3}
    c_3(w_2,w_5)+c_3(z,w_1)\geq 3\cdot \frac{k-1}{2}.
\end{align}

For each pair of vertices $v_1, v_2$ at distance $3$, by counting the edges between $C_3(v_1,v_2)$ and $C_3(v_2,v_1)$, we have $c_3(v_1,v_2) = c_3(v_2,v_1)$. 
Therefore, by \eqref{c3+c3}, we obtain
\[
c_3(w_5,w_2) + c_3(w_1,z)  \geq 3 \cdot \frac{k-1}{2}.
\]
Note that $C_3(w_5,w_2)\subseteq \Gamma_3(x)$ and $w_2$ has $\frac{k+1}{2}$ neighbours in $\Gamma_3(x)$.
Thus $c_3(w_5,w_2)\leq \frac{k+1}{2}$, and hence
\[
c_3(w_1,z)\geq 3\cdot \frac{k-1}{2}-\frac{k+1}{2}=k-2.
\]
As $C_3(w_1,z)\subseteq C_4(x,z)$, we obtain $c_4(x,z)\geq k-2$.
Therefore $c_4(x,z)\geq k-2$ for each $z\in \Gamma_4(x)$.

Since $b_3(x,y)\leq \frac{k-1}{2}$ for each $y\in \Gamma_3(x)$, $c_4(x,z)\geq k-2$ for each $z\in \Gamma_4(x)$, and $k_3\leq 2k$, by double counting the edges between $\Gamma_3(x)$ and $\Gamma_4(x)$ we obtain
\begin{align}\label{k4bound}
    2k\cdot \frac{k-1}{2}\geq k_4\cdot (k-2).
\end{align}

If $k_5 \geq 2$, then as two vertices in $\Gamma_5$ have at most $\mu$ common neighbours in $\Gamma_4$, we have $k_4 \geq k + (k-\mu) = k + \frac{k+1}{2}$. 
Combining this with \eqref{k4bound}, we obtain 
\[
k(k-1) \geq \left(\frac{3k+1}{2}\right)(k-2),
\] 
which is impossible as $k \geq 5$.

Thus, $k_5 = 1$. Since $\Gamma$ is a bipartite regular graph, we have $k_0 + k_2 + k_4 = k_1 + k_3 + k_5$. Given that $k_3 \leq k_2 = 2k$ and $k_4 \geq k$, we must have $k_3 = 2k$ and $k_4 = k$. Therefore, we obtain the distribution diagram in Figure~\ref{fig:AA2}. Each bipartite half now has $3k+1$ vertices. For any vertex $y$, the same edge count gives $|\Gamma_3(y)|\leq 2k$, so $|\Gamma_5(y)|=3k+1-k-|\Gamma_3(y)|\geq 1$. Thus the argument applies to every vertex, and $\Gamma$ is a distance-regular graph with intersection array $\left(k, k - 1, \frac{k+1}{2}, \frac{k-1}{2}, 1;\ 1, \frac{k-1}{2}, \frac{k+1}{2}, k - 1, k\right)$.

\begin{figure}[ht]
    \centering
    \includegraphics[width=\textwidth]{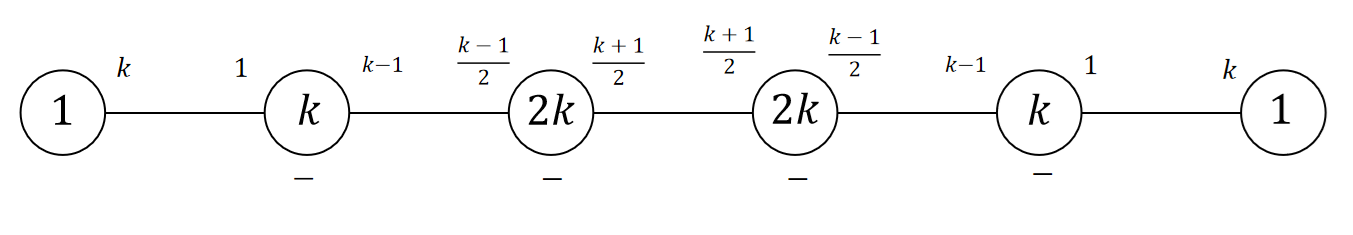} 
    \caption{The distribution diagram}
    \label{fig:AA2}
\end{figure}

The folded graph $\overline{\Gamma}$ is a strongly regular graph with parameters $\left(3k + 1, k, 0, \frac{k - 1}{2} \right)$. Its complement $\overline{\Gamma}^c$ is a strongly regular graph with parameters $\left(3k + 1, 2k, \frac{3}{2}(k - 1), k + 1\right)$. For a strongly regular graph with parameters $(v, k, \lambda, \mu)$, the nontrivial eigenvalues $\theta_1$ and $\theta_2$ (i.e., $\theta_i \neq k$) satisfy the relations $-\theta_1 \theta_2 = k - \mu$ and $\theta_1 + \theta_2 = \lambda - \mu$. Therefore, for $\overline{\Gamma}^c$, we have $\theta_1 = \frac{k - 1}{2}$ and $\theta_2 = -2$.

Since all strongly regular graphs with smallest eigenvalue $-2$ are known (see~\cite[p.5]{BM-2022}), it follows that $\overline{\Gamma}^c$ is the Clebsch graph with parameters $(16, 10, 6, 6)$. Thus, $k = 5$, and $\overline{\Gamma}$ is a strongly regular graph with parameters $(16, 5, 0, 2)$, i.e., the folded $5$-cube (see~\cite[Table~1.1]{BM-2022}). Hence, $\Gamma$ is the $5$-cube.

In conclusion, we have shown that a connected bipartite amply regular graph with diameter $d = 5$ and parameters $\left(v, k, 0, \frac{k - 1}{2}\right)$, where $k \geq 5$, must be the $5$-cube. Since the $5$-cube is not the bipartite double of any connected amply regular graph with parameters $(16, 5, 0, 2)$ and diameter $5$, Theorem~\ref{diameter5} follows.
\qed

\vspace{0.2cm}

We now consider the case $d = 4$.

\begin{theo}\label{diameter4}
Let $\Gamma$ be a connected amply regular graph with diameter $d = 4$ and parameters $(v, k, 0, \mu)$, where $\mu = \frac{k - 1}{2}$ and $k \geq 5$ is odd. Let $\Lambda$ be the unique bipartite $(0,2)$-graph with $14$ vertices and valency $4$, which is the point--block incidence graph of a square $2$-$(7,4,2)$ design.
Then either $\Gamma$ is isomorphic to $\K_2 \square \Lambda$, or $\Gamma$ is a bipartite $Q$-regular graph with distribution diagram shown in Figure~\ref{fig:D41}.
\end{theo}

\proof
Suppose $\Gamma$ is not bipartite. Then, by Lemmas~\ref{lambda=0} and~\ref{bipartite-double}, the bipartite double $\Gamma'$ of $\Gamma$ is a connected amply regular graph with diameter $d(\Gamma') = 4$ or $5$ and parameters $\left(2v, k, 0, \frac{k - 1}{2}\right)$. If $d(\Gamma') = 5$, then by Theorem~\ref{diameter5}, $\Gamma'$ must be the 5-cube. However, the 5-cube is not the bipartite double of any amply regular graph with diameter $4$ and parameters $(16, 5, 0, 2)$, a contradiction (see~\cite[Table 2]{Brouwer-2006}).
Assume that $d(\Gamma') = 4$. Since $d(\Gamma) = 4$, there exist vertices $x, y \in V(\Gamma)$ such that $d_{\Gamma}(x, y) = 4$. Clearly, $d_{\Gamma'}(x^+, y^-) \geq 4$, which must be odd—a contradiction. Thus, $\Gamma$ must be bipartite.

From now on, assume that $\Gamma$ is a connected bipartite amply regular graph with diameter $d = 4$ and parameters $\left(v, k, 0, \frac{k - 1}{2}\right)$, where $k \geq 5$ is odd.

Fix a vertex $x \in V(\Gamma)$, and define $k_i = |\Gamma_i(x)| = |\{ y \in V(\Gamma) \mid d_\Gamma(x, y) = i \}|$ for $0 \leq i \leq 4$. Let $\mathcal{B} = \{\Gamma_0(x), \Gamma_1(x), \Gamma_2(x), \Gamma_3(x), \Gamma_4(x)\}$ be a partition of $V(\Gamma)$. For convenience, set $k_4 = \alpha$. Since $k_2=2k$ and both bipartite halves have size $v/2$, $\alpha=v/2-1-2k$ is independent of $x$ and is positive because $d=4$. Since $\Gamma$ is a bipartite regular graph, we have $k_0 + k_2 + k_4 = k_1 + k_3$, and hence $k_3 = k_2 + \alpha + 1 - k$. This yields the distribution diagram shown in Figure~\ref{fig:D42}.

\begin{figure}[ht]
    \centering
    \includegraphics[width=\textwidth]{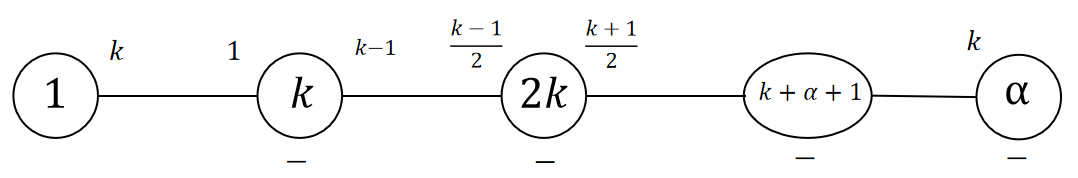} 
    \caption{The distribution diagram}
    \label{fig:D42}
\end{figure}

For each $y \in \Gamma_3(x)$, by~\cite[Proposition~1.9.1]{BCN}, we have $c_3(x,y) \geq c_2 + 1 = \frac{k + 1}{2}$. Since $k_3 = k + \alpha + 1 \leq k_2 = 2k$, it follows that $\alpha \leq k - 1$. We now prove that $\alpha = 1$ unless $k = 5$ and $\Gamma \cong \K_2 \square \Lambda$.

Suppose $\alpha \geq 2$. Let $X = \Gamma_3(x)$ and $Y = \Gamma_4(x)$. By Lemma~\ref{mubound},

\begin{align}
\label{eqd4-1}
\frac{k - 1}{2} = \mu &\geq \frac{k}{k + \alpha + 1} \left( k - \frac{\alpha + 1}{\alpha - 1} \right) = \frac{k}{k + \alpha + 1} \left( k - \frac{2}{\alpha - 1} - 1 \right).
\end{align}

Therefore, 
\begin{subequations}
\begin{align*}
& (k - 1)(k + \alpha + 1) \geq 2k\left(k - \frac{2}{\alpha - 1} - 1\right) \geq 2k(k - 3), \\
\implies\quad & (\alpha + 1)(k - 1) \geq 2k(k - 3) - k(k - 1) = k(k - 5), \\
\implies\quad & \alpha + 1 \geq k - 4, \\
\implies\quad & \alpha \geq k - 5.
\end{align*}
\end{subequations}

Set $\alpha = k - i$ for $1 \leq i \leq 5$. Then inequality~\eqref{eqd4-1} becomes

\begin{subequations}
\begin{align*}
& \frac{k - 1}{2} \geq \frac{k}{2k + 1 - i} \left( k - \frac{2}{k - i - 1} - 1 \right), \\
\iff\quad & (k - 1)(2k + 1 - i) \geq 2k\left( k - \frac{2}{k - i - 1} - 1 \right), \\
\iff\quad & \frac{4k}{k - i - 1} + i - 1 \geq (i - 1)k.
\end{align*}
\end{subequations}

By direct calculation, we obtain all possible solutions as follows:
\begin{itemize}
\item[(1)] $i = 1$, $k \geq 5$, $\alpha = k - 1$,
\item[(2)] $i = 2$, $k = 5$ or $7$, $\alpha = k - 2$,
\item[(3)] $i = 3$, $k = 5$, $\alpha = k - 3$,
\item[(4)] $i = 5$, $k = 7$, $\alpha = k - 5$.
\end{itemize}

We first eliminate Case~(1). Otherwise, the distribution diagram would be as in Figure~\ref{fig:D43}.

\begin{figure}[ht]
    \centering
    \includegraphics[width=\textwidth]{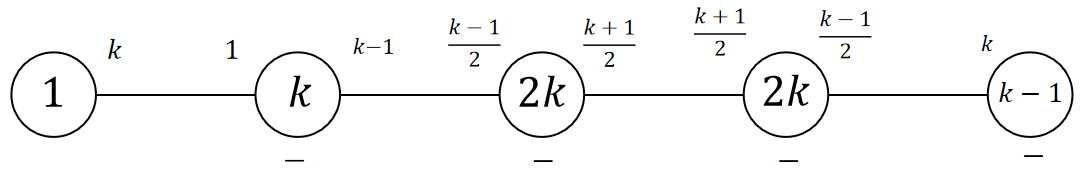} 
    \caption{The distribution diagram}
    \label{fig:D43}
\end{figure}

Let $X = \Gamma_3(x)$ and $Y = \Gamma_4(x)$. By Lemma~\ref{equalitycase}, since $(k - 2) \cdot \frac{k - 1}{2} > k \left( \frac{k - 1}{2} - 1 \right)$, there exist $y_1, y_2 \in \Gamma_4(x)$ such that $d(y_1, y_2) = 4$. So $\Gamma(y_1) \cup \Gamma(y_2) = \Gamma_3(x)$ and $\Gamma(y_1) \cap \Gamma(y_2) = \emptyset$. For $z \in \Gamma_4(x) \setminus \{y_1, y_2\}$, we have $\Gamma(z) \subseteq \Gamma(z, y_1) \cup \Gamma(z, y_2)$, hence $k = |\Gamma(z)| \leq |\Gamma(z, y_1)| + |\Gamma(z, y_2)| = 2\mu = k - 1$, a contradiction.

If $k = 5$, then by~\cite[Table 1]{Brouwer-2006}, $\Gamma$ is either the bipartite double of the icosahedron or $\K_2 \square \Lambda$.  
It is straightforward to verify that the bipartite double of the icosahedron satisfies $\alpha = 1$.

Consider $k = 7$. 
Now, each pair of vertices in $\Gamma_4(x)$ is at distance $2$ and hence shares $\mu=3$ common neighbors in $\Gamma_3(x)$.
Hence $k_3 = 8 + \alpha \geq (k - \mu) + k = 11$, implying $\alpha \geq 3$, and thus $\alpha \neq 2$.

If $\alpha = 5$, the corresponding distribution diagram is shown in Figure~\ref{fig:D44}.

\begin{figure}[ht]
    \centering
    \includegraphics[width=\textwidth]{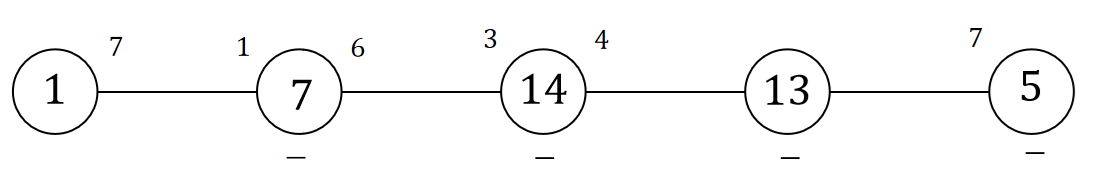} 
    \caption{The distribution diagram }
    \label{fig:D44}
\end{figure}

Let $\Gamma_3(x) = \{y_1, y_2, \ldots, y_{13}\}$ and define $n_i = |\Gamma(y_i) \cap \Gamma_4(x)|$ for $1 \leq i \leq 13$. By~\cite[Proposition~1.9.1]{BCN}, for all $y \in \Gamma_3(x)$, $c_3(x, y) \geq c_2 + 1 = 4$, hence $n_i \leq 3$. Since $\sum_{i=1}^{13} n_i = 35 = 3 \cdot 13 - 4$, and every pair $z_1, z_2 \in \Gamma_4(x)$ is at distance 2, the number
\[
|W| = \left|\left\{(y, \{z_1, z_2\}) \mid y \in \Gamma_3(x),\ \{z_1,z_2\} \subseteq \Gamma_4(x),\ z_1\neq z_2,\ y \sim z_1,\ y \sim z_2 \right\}\right| = \sum_{i=1}^{13} \binom{n_i}{2}
\]
equals $\binom{5}{2} \cdot \mu = 30$. Therefore, at most nine of the $n_i$ equal 3. Since $\sum_{i=1}^{13} n_i = 3 \cdot 13 - 4 = 35$, exactly nine of the $n_i$  equal 3. In that case,
\[
\sum_{i=1}^{13} \binom{n_i}{2} = 9 \cdot \binom{3}{2} + 4 \cdot \binom{2}{2} = 31,
\]
a contradiction.

Thus, $\alpha = 1$ always holds, except when $k = 5$ and $\Gamma \cong \K_2 \square \Lambda$.  
In the case $\alpha = 1$, it is straightforward to verify that $\Gamma$ admits an equitable partition whose distribution diagram is shown in Figure~\ref{fig:D41}. Since the above argument applies to every vertex $x \in V(\Gamma)$, it follows that $\Gamma$ is $Q$-regular. This completes the proof of Theorem~\ref{diameter4}.
\qed

We are now ready to prove Theorem~\ref{ARmain}.

\medskip
\noindent {\bf Proof of Theorem~\ref{ARmain}.}
By Lemma~\ref{lambda=0}, we have $\lambda = 0$ and $d \in \{4, 5\}$.  
If $d = 5$, then Theorem~\ref{diameter5} implies that $\Gamma$ is the $5$-cube.  
If $d = 4$, then by Theorem~\ref{diameter4}, $\Gamma$ is either isomorphic to $\K_2 \square \Lambda$ or is a bipartite $Q$-regular graph with distribution diagram shown in Figure~\ref{fig:D41}.
For the latter case, Theorem~\ref{thm:equivalent} shows that $\Gamma$ is the point--block incidence graph of a $GDDDP\left(2, k+1;\, k;\, 0, \frac{k-1}{2}\right)$, and hence satisfies item~(3) of Theorem~\ref{ARmain}.

This completes the proof of Theorem~\ref{ARmain}.
\qed

\vspace{0.4cm}

At the end of this section, we present a characterization of connected amply regular graphs with diameter $d \geq 2$ and parameters $k=7$ and $\mu=3$, which is the case we need to consider in~\cite{JKL-2025-2}.

\begin{cor}\label{valency7}
Let $\Gamma$ be a connected amply regular graph with diameter $d \geq 2$ and parameters $k = 7$ and $\mu = 3$. Then $\Gamma$ is bipartite, $\lambda=0$, and either $d=3$ and $v=30$, or $d=4$ and $v=32$.
\end{cor}

\proof
Let $x\in V(\Gamma)$. The local graph $[\Gamma(x)]$ is $\lambda$-regular on $7$ vertices, so $\lambda$ is even. Counting edges between $\Gamma_1(x)$ and $\Gamma_2(x)$ gives $7(6-\lambda)=3|\Gamma_2(x)|$. Since $d\geq 2$, we have $0\leq\lambda<6$, and hence $\lambda=0$ and $|\Gamma_2(x)|=14$.

Suppose that $\Gamma$ is not bipartite. By Lemma~\ref{bipartite-double}, its bipartite double $\Gamma'$ is a connected amply regular graph with parameters $(2v,7,0,3)$. 
Since $\lambda=0$, the graph $\Gamma$ is triangle-free.
Any path from $x^+$ to $x^-$ in $\Gamma'$ has odd length,
and its underlying vertices form a closed walk from $x$
to itself in $\Gamma$. Such a walk has length at least $5$.
Hence $d(\Gamma')\geq 5$.
Theorem~\ref{ARmain} gives $|V(\Gamma')|=32$, and hence $v=16$, contradicting $v\geq 1+7+14=22$. Thus $\Gamma$ is bipartite, and $d\neq 2$, since a connected bipartite graph of diameter $2$ is complete bipartite and would have $\mu=7$.

If $d\geq 4$, then Theorem~\ref{ARmain} yields
$d=4$ and $v=32$, with the distribution diagram
shown in Figure~\ref{fig:BB1}.

\begin{figure}[ht]
    \centering
    \includegraphics[width=\textwidth]{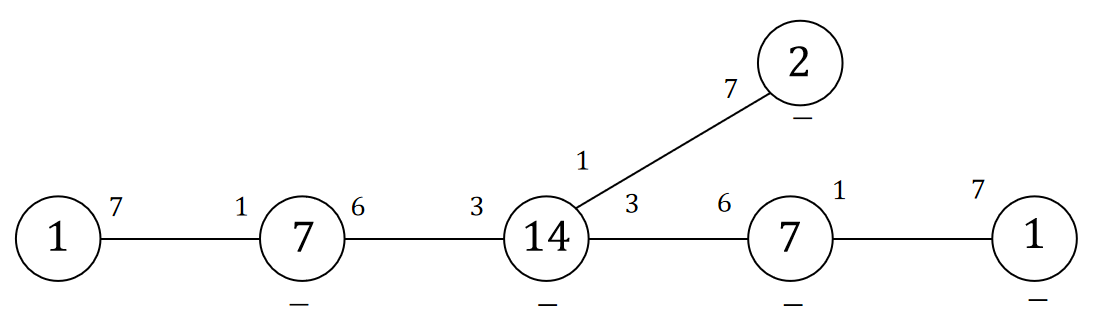} 
    \caption{The distribution diagram}
    \label{fig:BB1}
\end{figure}

If $d=3$, the bipartite half containing $x$ is $\{x\}\cup\Gamma_2(x)$ and has $15$ vertices. Both halves have the same size, so $v=30$. The distance partition is equitable with intersection array $\{7,6,4;1,3,7\}$ and distribution diagram shown in Figure~\ref{fig:BB2}.

\begin{figure}[ht]
    \centering
    \includegraphics[scale=0.8]{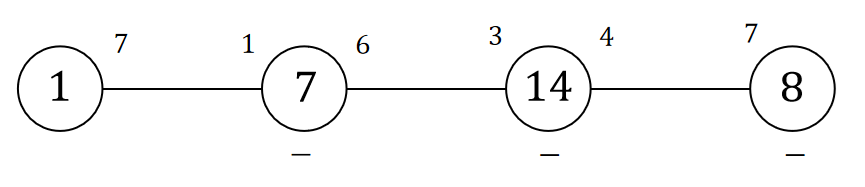} 
    \caption{The distribution diagram }
    \label{fig:BB2}
\end{figure}

This completes the proof of Corollary~\ref{valency7}.
\qed

\begin{remar}
The case $d=3$ in Corollary~\ref{valency7} does occur.
Indeed, the point--hyperplane incidence graph of
$\mathrm{PG}(3,2)$ is a bipartite distance-regular graph
with intersection array $\{7,6,4;1,3,7\}$.
\end{remar}

\section*{Acknowledgements}
Wei Jin is supported by NSFC (12271524, 12331013) and the NSF of Jiangxi \\ (20224ACB201002). Jack H. Koolen is supported by NSFC (12471335) and the Anhui Initiative in Quantum Information Technologies (AHY150000).

\end{document}